\documentclass[a4paper,11pt,leqno]{amsart}

\usepackage{amssymb}
\usepackage{vmargin}
\setmarginsrb{25mm}{25mm}{25mm}{20mm}{0pt}{0pt}{0pt}{12mm}

\newtheorem{theorem}{Theorem}[section]
\newtheorem{proposition}[theorem]{Proposition}
\newtheorem{lemma}[theorem]{Lemma}
\newtheorem{corollary}[theorem]{Corollary}

\newcommand{\N}{\mathbb N}

\newcommand{\R}{\mathbb R}
\newcommand{\C}{\mathbb C}
\newcommand{\ordd}[3]{d({#1},{#3}) = d( \{ {#1},{#2},{#3} \} )}

\newcommand{\om}[1]{\mathcal{O}({#1})}
\newcommand{\oom}[1]{\mathcal{O}_o({#1})}
\newcommand{\Hm}{{\mathcal{H}^1}}
\newcommand{\NS}{\mathcal{N}}
\newcommand{\DP}{\mathcal{D}}
\newcommand{\A}{\mathcal{A}}
\newcommand{\G}{\mathcal{G}}
\newcommand{\TK}{\mathcal{T}_K}
\newcommand{\m}{{\mu}}
\newcommand{\E}{\mathcal{E}}

\def\Xint#1{\mathchoice
   {\XXint\displaystyle\textstyle{#1}}%
   {\XXint\textstyle\scriptstyle{#1}}%
   {\XXint\scriptstyle\scriptscriptstyle{#1}}%
   {\XXint\scriptscriptstyle\scriptscriptstyle{#1}}%
   \!\int}
\def\XXint#1#2#3{{\setbox0=\hbox{$#1{#2#3}{\int}$}
     \vcenter{\hbox{$#2#3$}}\kern-.5\wd0}}
\def\dashint{\Xint-}

\begin{document}

\title{Menger curvature and rectifiability in metric spaces}
\author{Immo Hahlomaa}
\address{Department of Mathematics and Statistics, P.O. Box 35 (MaD), FI-40014 University of Jy\-v\"as\-ky\-l\"a, Finland}
\email{immo.a.hahlomaa@jyu.fi}
\date{}
\pagestyle{plain}

\begin{abstract}
We show that for any metric space $X$ the condition
\[ \int_X\int_X\int_X c(z_1,z_2,z_3)^2\, d\Hm z_1\, d\Hm z_2\, d\Hm z_3 < \infty, \]
where $c(z_1,z_2,z_3)$ is the Menger curvature of the triple $(z_1,z_2,z_3)$,
guarantees that $X$ is rectifiable. 
\end{abstract}

\maketitle

\section{Introduction}

Throughout the paper $(X,d)$ is a metric space.
Let $z_1$, $z_2$ and $z_3$ be three points of $X$.
The \emph{Menger curvature} of the triple $(z_1,z_2,z_3)$ is
\[ c(z_1,z_2,z_3) = \frac{2\sin\sphericalangle z_1z_2z_3}{d(z_1,z_3)}, \]
where
\[ \sphericalangle z_1z_2z_3 = \arccos \frac{d(z_1,z_2)^2+d(z_2,z_3)^2-d(z_1,z_3)^2}{2d(z_1,z_2)d(z_2,z_3)}. \]
Note that $c(z_1,z_2,z_3)$ is the reciprocal of the radius of the circle passing through
$x_1$, $x_2$ and $x_3$ whenever
$\{x_1,x_2,x_3\} \subset \R^2$ is an isometric triple for $\{z_1,z_2,z_3\}$.
For $K \in [1,\infty]$, a Borel subset $Z\subset X$ and a Borel measure $\m$ on $X$ we set
\[ c^2_K(Z,\m) = \int_{\TK(Z)} c(z_1,z_2,z_3)^2\, d\m^3(z_1,z_2,z_3), \]
where
\[ \TK(Z) = \left\{\, (z_1,z_2,z_3) \in Z^3\, : \, \text{$d(z_i,z_j) < Kd(z_k,z_l)$ for all $i,j,k,l\in\{1,2,3\}$, $k\neq l$}\, \right\}. \]
We also write
$c^2_K(Z) = c^2_K(Z,\Hm)$
and
$c^2(Z) = c^2_\infty(Z,\Hm)$,
where $\Hm$ is the 1-dimen\-sio\-nal Hausdorff measure on $X$ (or on $Z$). 

The diameter of $Z$ is denoted by $d(Z)$ and $B(x,r)$ stands for the closed ball in $X$ with center $x\in X$ and radius $r>0$.
If $W \subset U \times V$ and $u \in U$, where $U$ and $V$ are any sets, we write $W_u = \{\, v \in V\, :\, (u,v) \in W\, \}$.
For $U_0 \subset U$, a measure $\m$ on $U$ and a function $f:U_0 \to \overline{\R}$ we use the notation
\[ \dashint_{U_0} f\, d\m = \frac{1}{\m(U_0)}\int_{U_0} f\, d\m \]
if the right-hand side is defined.
We say that a metric space $X$ is \emph{rectifiable} if there is $E\subset\R$ and a Lipschitz function $f:E\to X$ such that
$\Hm\left(X\backslash f(E)\right) = 0$.

In this paper we will prove the following theorem.

\begin{theorem}
\label{thm1}
  If $X$ is a metric space with $c^2(X) < \infty$ then $X$ is rectifiable.
\end{theorem}

It was already known that any Borel set $X\subset\R^n$ with $\Hm(X)<\infty$ and $c^2(X)<\infty$
is rectifiable.
This was first proved by David in an unpublished paper.
In \cite{MR1709304} L\'eger gave a different proof.
Further a very different proof in the case $n=2$ has been given by Tolsa in \cite{MR2164419}.
The proof of Theorem~\ref{thm1} given here follows the ideas of David's proof.
As a matter of fact, the basic idea and some parts of our proof are taken quite directly from it.
This result was a part of the argument,
when David proved in \cite{MR1654535} that a purely unrectifiable set in $\C$ with finite length measure is
removable for bounded analytic functions.
Under the additional assumption that the set is 1-Ahlfors-regular this was already proved by
Mattila, Melnikov and Verdera in \cite{MR1405945}. Also they used the curvature by showing that $E\subset\C$ is
contained in an Ahlfors-regular curve if there is $C<\infty$ such that
$c^2(E \cap D) \leq Cd(D)$ for every disc $D$ in $\C$.
In \cite{MR2297880} we showed that for a bounded 1-Ahlfors-regular metric space $X$
the condition $c^2_K(X) < \infty$, where $K$ is a universal constant large enough,
implies that $X$ is a Lipschitz image of a bounded subset of $\R$.
More precisely, in this case one can find
$E\subset [0,1]$ and a Lipschitz surjection $f:E\to X$ with Lipschitz constant less than
$C(c^2_K(X) + d(X))$, where the constant
$C$ depends only on the 1-Ahlfors-regurality constant of $X$.
Recall that the 1-Ahlfors-regularity of $X$ means that there exists a constant $C<\infty$
such that
$C^{-1}r \leq \Hm(B(x,r)) \leq Cr$
whenever $x\in X$ and $r\in ]0,d(X)]$.


Most of this article will be spent on proving the following proposition.

\begin{proposition}
\label{prop}
  For any positive numbers $\mu_0$, $C_0$ and $\tau_0$ there exist $K<\infty$ and $\varepsilon_0>0$
  such that if $X$ is a separable metric space and $\mu$ is a Borel measure on $X$ verifying
  \begin{itemize}
    \item[\rm(i)] $\m(X) \geq \mu_0d(X)$,
    \item[\rm(ii)] $\m(B(x,r)) \leq C_0r$ for any $x\in X$ and $r>0$,
    \item[\rm(iii)] $c^2_K(X,\m) \leq \varepsilon_0d(X)$,
  \end{itemize}
  then there is $E\subset [0,1]$ and a Lipschitz function $f:E\to X$ such that the
  Lipschitz constant of $f$ is at most $(1+\tau_0)d(X)$ and
  $\m\left(X \backslash f(E)\right) \leq \tau_0d(X)$.
\end{proposition}

For any $\phi \in [0,1]$ we denote by $\om{\phi}$ the set of the metric
spaces $X$ for which $d(x,z)\geq d(x,y) + \phi d(y,z)$ whenever $x,y,z\in X$ are such that
$\ordd{x}{y}{z}$.
Notice that $\{x,y,z\} \in \om{\phi}$ whenever $\cos\sphericalangle xyz \leq -\phi \leq 0$.
We say that a metric space $X$ \emph{is orderable}, if there is an injection $o:X\to\R$ such that
for all $x,y,z\in X$ the condition $o(x)<o(y)<o(z)$ implies $d(x,z) > \max\{d(x,y),d(y,z)\}$.
In that case the function $o$ is called an \emph{order}.
If there is an order $o$ on $\{x_1,\dotsc,x_m\}$, $m\in\N$, such that
$o(x_i)<o(x_{i+1})$ for all $i\in\{1,\dotsc,m-1\}$, we write shortly $x_1x_2\dotsc x_m$.
We also denote $\oom{\phi} = \{\, X \in \om{\phi}\, :\, \text{$X$ is orderable}\, \}$.
The proof of the next lemma
can be found in \cite{MR2163108}.

\begin{lemma}
\label{lef}
  For any $L\geq 1$ there is $\phi < 1$ such that if $Z \in \om{\phi}$
  and $d(x,y) < Ld(z,w)$ for all $x,y,z,w\in Z$, $z\neq w$,
  then $Z$ is orderable or
  $Z = \{v_1,v_2,v_3,v_4\}$ with $v_1v_2v_3$, $v_2v_3v_4$, $v_3v_4v_1$ and $v_4v_1v_2$.
\end{lemma}

The following very simple lemma (see \cite{MR2297880}) will also be used later.

\begin{lemma}
\label{lemove}
  Let $\{x,y,z,z_1\}$ be a metric space such that $\{x,y,z\},\{x,y,z_1\}\in \om{\phi}$.
  \begin{itemize}
    \item[\rm(i)] If $xyz$ and
    $d(z,z_1) < \phi\min\{ d(x,y),d(y,z)+d(y,z_1) \}$, then $xyz_1$.
    \item[\rm(ii)]  If $xzy$ and
    $d(z,z_1) < \phi\min\{ d(z,x)+d(z_1,x), d(z,y)+d(z_1,y) \}$, then $xz_1y$.
  \end{itemize}
\end{lemma}

The next lemma is very useful in the proof of Proposition~\ref{prop}.

\begin{lemma}
\label{led}
  For any $\eta>0$ there are positive numbers $\eta_1$ and $\eta_2$ such that the following is true:
  Let $X$ be a metric space, $\m$ a Borel measure on $X$, $\delta \in [0,\infty[$and $r \geq d(X)$. If $\m(X)\geq \delta r$ and $c^2_5(X,\m) \leq \eta_1\delta^3 r$,
  then $\m(B(x,\eta r)) \geq \eta_2 \m(X)$ for some $x \in X$.
\end{lemma}

\begin{proof}
Fix $K=5$ and let $\phi \in ]3/5,1[$ be some fixed constant. We assume that $r\m(X) \in ]0,\infty[$.
The case $\m(X) = \infty$ can be treated similarly.
Choose $u_1\in X$ such that
\[ \m(X)\int_{\TK(X)_{u_1}} c(u_1,y,z)^2\, d\m^2(y,z) \leq c^2_K(X,\m), \]
and set $A_k = B(u_1,r\lambda^{k-1}) \backslash B(u_1,r\lambda^k)$ for $k\in\N$, where
$\lambda \in ](2\phi)^{-1},(K-2)(K\phi)^{-1}]$. 
Further let $k_0$ be the smallest integer such that $\lambda^{k_0} \leq a$,
where $a = 1 - \phi\lambda \in [2/K,1/2[$.
We shall show that there exists $Z\subset X$ such that $d(Z) \leq 2ar$ and $\m(Z) \geq (3k_0 + 1)^{-1}\m(X)$.
The desired result follows easily from this.

Let us denote $b=(3k_0 + 1)^{-1}$ and assume that $\m(B(u_1,ar)) < b\m(X)$. Then $\m(A_k) \geq 3b\m(X)$ for some $k \in \{1,2,\dotsc,k_0\}$.
We now choose $u_2\in A_k$ such that
\begin{gather}
  \m(X)\m(A_k)\int_{\TK(X)_{(u_1,u_2)}} c(u_1,u_2,z)^2\, d\m z \leq 2c^2_K(X,\m),  \label{b1}  \\
  \m(A_k)\int_{\TK(X)_{u_2}} c(x,u_2,z)^2\, d\m x\, d\m z \leq 2c^2_K(X,\m).  \notag
\end{gather}
We can assume that $\m(A_k\backslash B(u_2,ar)) \geq 2b\m(X)$. Since $Ka \geq 2$, we can choose $u_3\in A_k\backslash B(u_2,ar)$ such that
\begin{gather}
  \m(X)\m(A_k\backslash B(u_2,ar))\int_{\TK(X)_{(u_1,u_3)}} c(u_1,y,u_3)^2\, d\m y \leq 3c^2_K(X,\m),  \label{c1}  \\
  \m(A_k)\m(A_k\backslash B(u_2,ar))\int_{\TK(X)_{(u_2,u_3)}} c(x,u_2,u_3)^2\, d\m x \leq 6c^2_K(X,\m),  \label{c2}  \\
  \m(X)\m(A_k)\m(A_k\backslash B(u_2,ar)) c(u_1,u_2,u_3)^2 \leq 6c^2_K(X,\m).  \label{c3}
\end{gather}
  
Denote $F = \{\, w \in A_k\, :\, \{w,u_1,u_2,u_3\} \in \om{\phi}\, \}$.
We next show that $F \subset B(u_2,ar) \cup B(u_3,ar)$. For this,
assume that $w_1,w_2,w_3 \in A_k$ are distinct points such that $\{w_1,w_2,w_3,u_1\} \in \om{\phi}$, and denote $d_i = d(w_i,u_1)$ and $d_{ij} = d(w_i,w_j)$
for $i,j \in \{1,2,3\}$. Now $\ordd{w_i}{w_j}{u_1}$ for some distinct $i,j \in \{1,2,3\}$, because else we would have, by assuming
$d_1 \leq d_2 \leq d_3$, that
$d_{pq} + \phi d_{qr} - d_{pr} \geq d_2 + \phi d_1 +\phi(d_3 + \phi d_1) - d_2 - d_3 = (\phi-1)d_3 + \phi(1+\phi)d_1 > (\phi-1+\phi(1+\phi)\lambda)r\lambda^{k-1} \geq 0$
for every distinct $p,q,r \in \{1,2,3\}$, which is a contradiction.
Further, if $\ordd{w_i}{w_j}{u_1}$ then $d_{ij} < (1 - \phi\lambda)r\lambda^{k-1} \leq ar$. Therefore, since $d(u_2,u_3) > ar$, we have
$F \subset B(u_2,ar) \cup B(u_3,ar)$.

Since $b^2r^2c(u_1,u_2,u_3)^2 \leq \eta_1$ by \eqref{c3}, we have $\{u_1,u_2,u_3\} \in \om{\sqrt{1-4^{-1}b^{-2}\eta_1}} \subset \om{\phi}$
by assuming $\eta_1 \leq 3b^2(1-\phi^2)$. Thus $A_k \backslash F \subset F_{12} \cup F_{13} \cup F_{23}$, where
$F_{ij} = \{\, w \in A_k\, :\, \{u_i,u_j,w\} \not\in \om{\phi}\, \}$.
Now by \eqref{b1}
\[ \m(F_{12}) \leq \int_{F_{12}} \frac{r^2c(u_1,u_2,w)^2}{4(1-\phi^2)}\, d\m w \leq \frac{\eta_1\delta r}{6b(1-\phi^2)} \]
and from \eqref{c1} and \eqref{c2} we similarly get
$ \m(F_{13}) \leq 3\eta_1\delta r(8b(1-\phi^2))^{-1}$ and $\m(F_{23}) \leq \eta_1\delta r(4b^2(1-\phi^2))^{-1}$.
Thus by taking $\eta_1 \leq b^3(1-\phi^2)$ we have
\[ \m(A_k \backslash (B(u_2,ar) \cup B(u_3,ar))) < \frac{\eta_1\delta r}{b^2(1-\phi^2)} \leq b\delta r \leq b\m(X) \]
and further $\max\{\, \m(B(u_2,ar)),\m(B(u_3,ar))\, \} \geq b\m(X)$.
\end{proof}

Denote
\[ \partial(z_1,z_2,z_3) = \min_{\sigma \in S_3} \left( d(z_{\sigma(1)},z_{\sigma(2)})+d(z_{\sigma(2)},z_{\sigma(3)})-d(z_{\sigma(1)},z_{\sigma(3)}) \right), \]
where $S_3$ is the set of permutations on $\{1,2,3\}$.
For Borel subset $Z \subset X$ we set
\[ \beta(Z) = \int_{Z^3} \frac{\partial(z_1,z_2,z_3)}{d(\{z_1,z_2,z_3\})^3}\, d(\Hm)^3(z_1,z_2,z_3). \]
One easily sees (\cite[Lemma 5.1]{MR2163108}) that for any $K \in [1,\infty[$
\begin{equation}
\label{cp}
  \frac{c^2_K(Z)}{4K^2} \leq \beta(Z) \leq \frac{c^2(Z)}{2}.
\end{equation}

\begin{lemma}
\label{lesigma}
  If $X$ is a metric space with $\beta(X) < \infty$ then
  the 1-dimensional Haus\-dorff measure on $X$ is $\sigma$-finite.
\end{lemma}

\begin{proof}
We can assume that $X$ is bounded. As in the proof of the previous lemma, we find $x_0 \in X$ such that for any $\lambda \in ]2^{-1/2},1[$
and $k\in\N$ there are Borel sets $F^1_{\lambda,k}$, $F^2_{\lambda,k}$ and $F^3_{\lambda,k}$ such that
$B(x_0,\lambda^{k-1}d(X)) \backslash B(x_0,\lambda^kd(X))= F^1_{\lambda,k} \cup F^2_{\lambda,k} \cup F^3_{\lambda,k}$,
where $d(F^1_{\lambda,k}),d(F^2_{\lambda,k}) \leq 2(1-\lambda^2)\lambda^{k-1}d(X)$
and $\Hm(F^3_{\lambda,k}) < \infty$. Taking a sequence $\lambda_j \uparrow 1$ we have 
\[ F := X \backslash \bigl( \{x_0\} \cup \bigcup_{j,k\in\N} F^3_{\lambda_j,k} \bigr) \subset \bigcup_{k\in\N} F^1_{\lambda_i,k} \cup \bigcup_{k\in\N} F^2_{\lambda_i,k} \]
for all $i\in\N$. Since now $1-\lambda^2 \to 0$ and
\[ (1-\lambda^2) \sum_{k=1}^\infty \lambda^{k-1} = 1 + \lambda \to 2 \]
as $\lambda\uparrow 1$, we have $\Hm(F) \leq 8d(X) < \infty$.
\end{proof}

By \eqref{cp} the following theorem implies Theorem~\ref{thm1}.

\begin{theorem}
\label{thm2}
  If $X$ is a metric space with $\beta(X) < \infty$ then $X$ is rectifiable.
\end{theorem}

A minor modification of the following lemma can be found in \cite{MR1709304} where it is stated for a set in $\R^n$,
but the proof, which uses the density theorem and the Vitali covering theorem for Hausdorff measures,
works for any metric space.

\begin{lemma}
\label{le0}
  Let $X$ be a metric space with $0<\Hm(X)< \infty$ and $\beta(X) < \infty$.
  Then for all $\varepsilon > 0$ there is a Borel set $Z\subset X$ such that
  \begin{itemize}
    \item[\rm(i)] $\Hm(Z) > d(Z)/40$,
    \item[\rm(ii)] $\Hm(Z \cap B(z,r)) \leq 3r$ for any $z\in Z$ and $r>0$,
    \item[\rm(iii)] $\beta(Z) \leq \varepsilon d(Z)$.
  \end{itemize}
\end{lemma}

Taking Proposition~\ref{prop} for granted we can now give a proof of Theorem~\ref{thm2}
following \cite{MR1709304}.

\begin{proof}[Proof of Theorem~\ref{thm2}.]
Let $X$ be a metric space with $\beta(X) < \infty$. By Lemma~\ref{lesigma} we may assume $\Hm(X) < \infty$. Suppose to the contrary that $X$ is not rectifiable.
Then there is a subset $Y \subset X$ such that $\Hm(Y) > 0$ and $\Hm(Y \cap g(E)) = 0$ for each Lipschitz function
$g : E \to X$ with $E \subset \R$.
Let $\varepsilon_0$ and $K$ be as in Proposition~\ref{prop} depending on $\mu_0 = 1/40$, $C_0 = 3$ and $\tau_0 = 1/80$.
By Lemma~\ref{le0} we find $Z \subset Y$ so that $\beta(Z) \leq \varepsilon_0d(Z)/4K^2$, $\Hm(Z) > d(Z)/40$ and
$\Hm(Z \cap B(z,r)) \leq 3r$ for all $z \in Z$ and $r>0$.
Now $c^2_K(Z) \leq 4K^2\beta(Z) \leq \varepsilon_0d(Z)$
by \eqref{cp}.
Thus by Proposition~\ref{prop} we find a Lipschitz function $f : E \to Z$ such that
$E \subset [0,1]$ and $\Hm(Z\backslash f(E)) \leq d(Z)/80$.
Hence $\Hm(Z \cap f(E)) \geq d(Z)/80 > 0$, which is a contradiction.
\end{proof}

One can trivially replace $\beta(X)$ in Lemma~\ref{le0} by the integral $\int_{X^3} g\, d(\Hm)^3$ where $g : X^3 \to [0,\infty]$ is any Borel function.
Hence also the condition $c^2_K(X) < \infty$ and $\Hm(X) < \infty$ implies the rectifiability of $X$ provided that the constant $K$ is large enough.

We would like to say something about the converse results in general metric spaces.
The following theorem of Schul can be found in \cite{MR2337487}.


\begin{theorem}
\label{thmSchul}
  {\rm\cite{MR2337487}}
  Let $X$ be a connected 1-Ahlfors-regular metric space. Then
  \[ \beta(X) \leq C\Hm(X), \]
  where the constant $C$ depends only on the 1-Ahlfors-regurality constant of $X$.
\end{theorem}

Combining Theorem~\ref{thm2} with Theorem~\ref{thmSchul} and \cite[Theorem 1.1]{MR2554164} one obtains the following characterization of rectifiability.

\begin{corollary}
  A metric space $X$ is rectifiable if and only if $X$ can be written as
  \[ X = \bigcup_{i=1}^\infty X_i\quad\text{with $\beta(X_i) < \infty$ for all $i$}. \]
\end{corollary}

\section{Preliminaries of the proof of Proposition~\ref{prop}}

From now on we assume that the hypotheses of Proposition~\ref{prop} are satisfied.
Clearly we can assume that $0<d(X)<\infty$, since else the statement of Proposition~\ref{prop} is trivial.
By replacing $X$ by $\varphi(X)$, where $\varphi:X \to \ell^\infty (X)$ is the Kuratowski embedding, we can assume that $X$ is a subset of a Banach space $\NS$.
We will construct a sequence of curves $\Gamma_n$ in $\NS$ which approximate $X$.
Each $\Gamma_n$ will be obtained by choosing points $x_i \in X$, $i=1,2,\dotsc,k(n) \in \N$, and joining $x_i$ to $x_{i+1}$ by a line segment
in $\NS$ for each $i \in \{1,2,\dotsc,k(n)-1\}$.
We then show that the length of the curve $\Gamma_n$
\[ l(\Gamma_n) := \sum_{i=1}^{k(n)-1} d(x_i,x_{i+1}) \]
is uniformly bounded by $Ld(X)$, where $L < \infty$ is a constant depending on $\mu_0$, $\varepsilon_0$ and $C_0$.
In other words, we find a sequence of $Ld(X)$-Lipschitz surjections $f_n:[0,1] \to \Gamma_n$.
Since in our construction the closure of $\bigcup_n \Gamma_n$ is a compact subset of $\NS$,
we find by the Ascoli-Arzela theorem a $Ld(X)$-Lipschitz function $f:[0,1] \to \NS$, which is the uniform limit of some subsequence of $(f_n)$.
Finally we show that $\m(X \backslash \Gamma)$ is small. Here we denote $\Gamma = f([0,1])$.

We now describe how we choose the vertices for the curves $\Gamma_n$.
Let $n_0$ be the largest integer such that $d(X) \leq 2^{-n_0}$, and set $H_{n_0}=D_{n_0-1}=\emptyset$.
Let now $n \geq n_0$ and assume by induction that we have defined $H_n$ and $D_{n-1}$.
Denote
\[ \DP_n = \bigcap_{m=n_0}^{n+N_0} \{\, x \in X\, :\, \m(B(x,2^{-m})) \geq \delta 2^{-m}\, \}, \]
where $N_0 \in \N$ and $\delta>0$ are constants fixed later.
For any $x\in \DP_n$ we choose a point $q_n(x) \in B(x,2^{-n-N_0})$ such that
\begin{align}
\label{kaava1}
 \int_{\A_n(x)} c(z_1,z_2,q_n(x))^2\, d\m^2(z_1,z_2) \leq \dashint_{B(x,2^{-n-N_0})}\int_{\A_n(x)} c(z_1,z_2,z_3)^2\, d\m^2(z_1,z_2)\, d\m z_3,
\end{align} 
where
$\A_n(x) = \{\, (z_1,z_2) \in \left( B(x,R_1 2^{-n}) \backslash B(x,r_1 2^{-n}) \right)^2 \, :\, d(z_1,z_2) > r_1 2^{-n}\, \}$.
Here $R_1$ and $r_1$ are positive constants fixed later.
We set
\[ D_n = q_n(D_n'), \]
where $D_n'$ is a maximal subset of $\DP_n \backslash \bigcup_{y \in H_n} B(y)$ such that $d(z_1,z_2) > 2^{-n}$ for any distinct $z_1,z_2\in D_n'$.
For any $y \in H_n$ we write $B(y) = B(q_{m(y)}^{-1}(y),2^{-m(y)+3})$, where $m(y)$ is the largest integer $m$ such that $y \in D_m$.
We further set
\begin{equation}
\label{Hdef}
  H_{n+1} = H_n \cup \biggl\{\, x \in D_n\, :\, \m\bigl(B(q_n^{-1}(x),2^{-n+4}) \backslash \bigcup_{y \in H_n} B(y)\bigr) \leq C_1\delta 2^{-n}\, \biggr\},
\end{equation}
where $C_1 < \infty$ is a constant fixed later.
Denote $X_n = D_n \cup H_n$.
The curve $\Gamma_n$ is now determined by the set $X_n$ and an order on $X_n$.

Notice that $X_n$ is a finite subset of $X$, because $\mu(X)<\infty$ by (ii). Further
\begin{equation}
\label{d1}
  d(z_1,z_2) > (1-2^{-N_0+1})2^{-n} \geq 2^{-n-1}
\end{equation}
for any distinct $z_1,z_2\in X_n$ for all $n \geq n_0$. 
Since $D_{n+1}'  \subset \DP_n \backslash \bigcup_{y \in H_n} B(y)$ we trivially have
\begin{gather}
    d(x,D_n') \leq  2^{-n} \quad\text{for all $x \in D_{n+1}'$},  \label{d2}  \\
    d(x,D_n) \leq (1 + 2^{-N_0+1})2^{-n} < 2^{-n+1}\quad\text{for all $x \in D_{n+1}$}.  \label{d3}
\end{gather}
For any $x\in X$ and $r>0$ we set
\[ c^2(x,r) = \frac{c^2_K(B(x,r),\m)}{r}. \]
Let $\varepsilon_1 > 0$ and $Z = \{\, z \in X\, :\, \text{$c^2(x,r) > \varepsilon_1$ for some $r > 0$}\, \}$. Let us choose for each $z\in Z$ a number $r(z)$
such that $c^2(z,r(z)) > \varepsilon_1$. Now $Z \subset \bigcup_{z\in Z} B(z,r(z))$. By the $5r$-covering lemma we find a countable set $Z_1 \subset Z$
such that $Z \subset \bigcup_{z\in Z_1} B(z,5r(z))$ and $B(z_1,r(z_1)) \cap B(z_2,r(z_2)) = \emptyset$ for distinct $z_1,z_2\in Z_1$, and we get by (iii)
\begin{align*}
  \m(Z) &\leq \sum_{z\in Z_1} \m(B(z,5r(z))) \leq 5C_0\sum_{z\in Z_1} r(z)
  < \frac{5C_0}{\varepsilon_1}\sum_{z\in Z_1}  c^2_K(B(z,r(z)),\m)  \\
  &\leq \frac{5C_0\varepsilon_0d(X)}{\varepsilon_1} \leq \frac{\tau_0d(X)}{2}
\end{align*}
as long as $\varepsilon_1 \geq 10C_0\varepsilon_0\tau_0^{-1}$.
Thus we can without loss of generality assume that
\begin{equation}
\label{eps1}
  c^2(x,r) \leq \varepsilon_1
\end{equation}
for all $x \in X$ and $r > 0$. We will fix the constant $\varepsilon_1$ later.

\begin{lemma}
\label{le2}
  For each integer $n \geq n_0$, $d(x,X_{n+1}) < 2^{-n+5}$ for all $x \in X_n$.
\end{lemma}

\begin{proof}
Since $H_n \subset H_{n+1} \subset X_{n+1}$ we can assume that $x \in D_n$ and
\[ \m\bigl(B(q_n^{-1}(x),2^{-n+4}) \backslash \bigcup_{y \in H_n} B(y)\bigr) > C_1\delta 2^{-n}. \]
By choosing $\varepsilon_1$ small enough depending on $N_0$ and $C_1\delta$ and then using Lemma~\ref{led}
we find $z \in B(q_n^{-1}(x),2^{-n+4}) \backslash \bigcup_{y \in H_n} B(y)$ such that
\begin{equation}
\label{a1}
  \m(B(z,2^{-n-N_0-1})) > \eta C_1\delta 2^{-n},
\end{equation}
where $\eta > 0$ depends on $N_0$.
  
We next show that $z \in \DP_{n+1}$. If $n \leq n_0+N_0+6$ this follows directly from \eqref{a1}
provided that $C_1$ is big enough depending on $N_0$.
Let us now assume that $n > n_0+N_0+6$.
We first show that $z \in \DP_{n-N_0-6}$.
If this is not true then there is an integer $m\in [n_0,n-6]$ such that
$\m(B(z,2^{-m})) < \delta 2^{-m}$.
Using \eqref{d2} we find $w \in D_{m+5}$ such that
\begin{equation}
\label{a2}
  d(q_n^{-1}(x),q_{m+5}^{-1}(w)) \leq 2^{-m-4}.
\end{equation}
Thus
$d(z,q_{m+5}^{-1}(w)) \leq 2^{-n+4} + 2^{-m-4} \leq 2^{-m-1}$
and so
\[ \m(B(q_{m+5}^{-1}(w),2^{-m-1})) \leq \m(B(z,2^{-m})) < \delta 2^{-m}. \]
By choosing $C_1 \geq 32$ we have that $w \in H_{m+6}$.
From this we get $d(q_n^{-1}(x),q_{m+5}^{-1}(w)) > 2^{-m-2}$,
which contradicts \eqref{a2}. So we have $z \in \DP_{n-N_0-6}$.
This and \eqref{a1} give $z \in \DP_{n+1}$
provided that $C_1$ is big enough depending on $N_0$.

If $z \not\in B(y)$ for all $y \in H_{n+1}$ then
$d(x,D_{n+1}) \leq d(x,z) + d(z,D_{n+1}) \leq 2^{-n+4} + 2^{-n-N_0} + 2^{-n-1} + 2^{-n-1-N_0} < 2^{-n+5}$.
Else $z \in B(y)$ for some $y \in H_{n+1} \backslash H_{n}$ and we get
$d(x,H_{n+1}) \leq d(x,z) + d(z,H_{n+1}) \leq 2^{-n+4} + 2^{-n+3} + 2^{-n-N_0+1} < 2^{-n+5}$.
\end{proof}

Let $n > n_0$. 
Let us write
\[ D_n^* = \{\, x \in D_n\, :\, d(x,D_{n-1}) \leq \vartheta2^{-n+1}\, \} = \left\{ x_n^1,\dotsc,x_n^{j_n} \right\}, \]
where $j_n = \#D_n^*$ and $\vartheta$ is any fixed constant between $1/4$ and $1/3$.
We define $X_{n-1}^0 = X_{n-1}$ and inductively
\[ X_{n-1}^k = \left(X_{n-1}^{k-1} \backslash \{ p(x_n^k) \}\right) \cup \left\{ x_n^k \right\} \]
for $k=1,\dotsc,j_n$,
where $p(x)$ be the unique point in $X_{n-1}^{k-1}$ such that $d(x,p(x)) = d(x,X_{n-1}^{k-1})$.
Notice that $p(x) \in D_{n-1}$ for all $x\in D_n^*$, and the mapping $p:D_n^* \to D_{n-1}$
is injective by \eqref{d1}. This is because $N_0$ is chosen to be a large integer.
Further we denote $k_n = \#D_n$ and write
\[ D_n \backslash D_n^* = \left\{ x_n^{j_n+1},\dotsc,x_n^{k_n} \right\}, \]
where
\begin{equation}
\label{xn}
  d\left( x_n^k, X_{n-1}^{k-1} \right) = \max\left\{\, d\left( x, X_{n-1}^{k-1} \right)\, :\, x\in X_n\, \right\}, 
\end{equation}
and
\[ X_{n-1}^k = X_{n-1}^{j_n} \cup \left\{ x_n^{j_n+1},\dotsc,x_n^k \right\} \]
for $k = j_n+1,\dotsc,k_n$. 

For any $n \geq n_0$ and $z \in X_n \cup D_{n+1}$ we denote
\[  m_n(z) = \begin{cases}
      m(z)	& \text{if $z \in H_n$,}  \\
	n	& \text{if $z \in D_n \backslash D_{n+1}$,}  \\
	n+1	& \text{if $z \in D_{n+1}$}
  \end{cases}
\]
and
\[ B_n(z) = B(q_{m_n(z)}^{-1}(z),2^{-m_n(z)-N_0}). \]
Suppose that $n \geq n_0$, $k \in \{0,\dotsc,k_{n+1}\}$ and $z_1,z_2$ are distinct points in $X_n^k$.
Let also $w_i \in B_n(z_i)$ for $i = 1,2$.
By the construction
\begin{equation}
\label{Q}
  Q_1^{-1}d(z_1,z_2) < d(w_1,w_2) < Q_1d(z_1,z_2),
\end{equation}
where $Q_1 = \vartheta(\vartheta - 2^{-N_0+2})^{-1}$. 
The constants $M_1$ and $\phi_1 \in [0,1[$ which appear in the next lemma will be fixed later.

\begin{lemma}
\label{leordom}
  The set $B(x,2^{-n+M_1+1}) \cap X_n^k$ belongs to $\oom{\phi_1}$
  for all $x\in X$, $n \geq n_0$ and $k \in \{0,\dotsc,k_{n+1}\}$.
\end{lemma}

\begin{proof}
Let us denote $Z = B(x,2^{-n+M_1+1}) \cap X_n^k$.
Now $\vartheta <  2^nd(z_1,z_2) \leq 2^{M_1+2}$ for all distinct $z_1,z_2 \in Z$.
Assuming that $\#Z \geq 2$
the construction gives that $m_n(z) \geq n-M_1-2$ for all $z \in Z$ (see \eqref{d1}).

Choose $z_0 \in Z$ and denote $A = B(q_{m_0}^{-1}(z),2^{-m_0-2}) \backslash B(q_{m_0}^{-1}(z),\sigma2^{-m_0})$,
where $m_0 = m_n(z_0)$.
Now $\m(A) \geq (\delta 2^{-2}-C_0\sigma)2^{-m_0}$.
By choosing $\sigma > 0$ small enough depending on on $\delta$ and $C_0$ and taking $\varepsilon_1$ small enough depending on
$\delta$ and $N_0$, and then using Lemma~\ref{led}
we find $y_0 \in A$ such that
$\m(U_{y_0}) \geq \eta \m(A)$,
where $U_{y_0} = B(y_0,2^{-m_0-N_0}) \cap A$ and $\eta > 0$ depends on $N_0$.
Now $l < 2^{n}d(z_1,z_2) \leq L$ for all distinct $z_1,z_2 \in Z \cup \{y_0\}$,
where $l = \min\{2^{-2}(\sigma-2^{-N_0}),\vartheta-2^{-2}-2^{-N_0}\}$ and
$L = 2^{M_1+2} + 2^{M_1+1}$.
Furthermore, choosing $N_0$ big enough depending on $\sigma$ and $\vartheta$, for any $y \in U_{y_0}$ and
$w \in B_n(z)$, $z \in Z$,
\begin{align*}
  Q_2^{-1}d(y_0,z) < d(y,w) < Q_2d(y_0,z),
\end{align*}
where $Q_2 = \max\{ (\sigma+2^{-N_0+1})(\sigma-2^{-N_0})^{-1} , (\vartheta-2^{-2}-2^{-N_0})(\vartheta-2^{-2}-2^{-N_0+1})^{-1} \}$.

Suppose now to the contrary that $Z \cup \{y_0\} \supset \{z_1,z_2,z_3\} \not\in \om{\phi_1}$.
For $i = 1,2,3$ let $w_i \in U_{y_0}$ if $z_i = y_0$ and $w_i \in B_n(z_i)$ if $z_i \neq y_0$.
Denote $d_{ij} = d(z_i,z_j)$ and $d_{ij}' = d(w_i,w_j)$ for $i,j = 1,2,3$, and
assume that $\ordd{w_1}{w_2}{w_3}$ and $d_{12}\geq d_{23}$.
By denoting $Q = \max\{Q_1,Q_2\}$ (see \eqref{Q}) and choosing $N_0$ big enough depending on $\sigma$, $\vartheta$,
$M_1$ and $\phi_1$,
\begin{align*}
  \frac{d_{13}' - d_{12}'}{d_{23}'}
  \leq \frac{Qd_{13} - Q^{-1}d_{12}}{Q^{-1}d_{23}}
  \leq \frac{d_{13} - d_{12} + (Q^2-1)d_{13}}{d_{23}}
  < \phi_1 + (Q^2-1)L/l < 1.
\end{align*}
So we have
\begin{align*}
  c(z_1,z_2,z_3)^2 = \frac{(2\sin\alpha)^2}{d(z_1,z_3)^2}
  \geq \frac{4(1-\cos^2\alpha)}{(2^{M_1+3})^22^{-2n}}
  \geq \frac{1-\max\{\theta^2,1/4\}}{(2^{M_1+2})^22^{-2n}},
\end{align*}
where $\alpha = \sphericalangle z_1z_2z_3$ and
$\theta  = \phi_1 + (Q^2-1)L/l$.

By using the above estimates and choosing the constant $K$ big enough depending on
$l$, $L$ and $Q$ we deduce that
the number $c^2(x,2^{-n+M_1+2})$ is larger than some positive constant depending on
$\theta$, $M_1$, $\delta$ and $N_0$. Taking $\varepsilon_1$ small enough this contradicts
\eqref{eps1} and so $Z \cup \{y\} \in \om{\phi_1}$.
By Lemma~\ref{lef} we can choose $\phi_1<1$ depending only on $L/l$ such that
$Z$ is orderable. Notice that if $\#Z = 4$ then we can apply Lemma~\ref{lef}
for $Z \cup \{y_0\}$.
\end{proof}

\section{Construction and length of $\Gamma$}

Notice that $\DP_{n_0} \neq \emptyset$ by Lemma~\ref{led} provided that $\varepsilon_0$ and $\delta$
are small enough depending on $N_0$ and $\mu_0$. Thus $D_{n_0}$ consists of one point. Further
$X_n \neq \emptyset$ for all integers $n \geq n_0$ by Lemma~\ref{le2}.
We define
$ \Gamma_{n_0}^0 = D_{n_0}$ and $\E_{n_0}^0 = \emptyset$.
For any indices $n \geq n_0$ and $k \in \{0,\dotsc,\#X_{n+1}\}$ we will denote by $\E_n^k$ the set of the edges of the curve $\Gamma_n^k$.
So $\Gamma_n^k$ is determined by $\E_n^k$ unless $\E_n^k$ is empty and $\Gamma_n^k$ is reduced to one point.
We will also write for $y \in X_n^k$
\[ N_n^k(y) = \{\, w \in X_n^k\, :\, \{y,w\} \in \E_n^k\, \}. \]

Let now $n \geq n_0$ and $k\in\{0,\dotsc,j_{n+1}-1\}$, and assume by induction that we have already constructed a curve $\Gamma_n^k$
such that $X_n^k \subset \Gamma_n^k$ and the following hypothesis is satisfied:
\begin{itemize}
  \item[(H1)] If $z \in X_n$, $B(z,2^{-n+M_1}) \cap X_n = \{x_1,\dotsc,x_j\}$
  and $x_1x_2\dotsc x_j$, then
  $\{ x_i,x_{i+1} \} \in \E_n^0$ for all $i \in \{1,\dotsc,j-1\}$.
\end{itemize}
Notice that (H1) is trivially true for $n=n_0$.
We now construct a curve $\Gamma_n^{k+1}$
such that $X_n^{k+1} \subset \Gamma_n^{k+1}$.
Denote $x=x_{n+1}^{k+1}$ and let $y\in X_n$ be such that $d(x,y) = d(x,X_n)$.
We simply replace $y$ by $x$, i.e. we set
\[ \E_n^{k+1} = \left(\E_n^k \backslash \{\, \{y,w\}\, :\, w \in N_n^k(y)\, \} \right) \cup \{\, \{x,w\}\, :\, w \in N_n^k(y)\, \}. \]
Since
$d(z_1,z_2) > (1-2^{-N_0+1}-2\vartheta)2^{-n}$ for any distinct $z_1,z_2 \in X_n^k$
by \eqref{d1} and
$d(x,y) \leq \vartheta 2^{-n} \leq \phi_1(1-2^{-N_0+1}-2\vartheta)2^{-n}$
by our choice of the constants, we easily see by Lemma~\ref{leordom}, Lemma~\ref{lemove}
and (H1) that the following hypothesis will be satisfied for each $j \in \{0,\dotsc,j_{n+1}\}$:
\begin{itemize}
  \item[(H2)] If $z \in X_n^j$, $B(z,(2^{M_1}-1)2^{-n}) \cap X_n^j = \{y_1,\dotsc,y_l\}$
  and $y_1y_2\dotsc y_l$, then
  $\{ y_i,y_{i+1} \} \in \E_n^j$ for all $i \in \{1,\dotsc,l-1\}$.
\end{itemize}

We now want to estimate the difference $l(\Gamma_n^{k+1})-l(\Gamma_n^k)$ in certain cases.
If $\#N_n^k(y) = 1$ we will use the simple estimate
\begin{equation}
\label{est11}
  l(\Gamma_n^{k+1})-l(\Gamma_n^k) \leq d(x,y)  \leq \vartheta 2^{-n}.
\end{equation}
Let us now assume that $\#N_n^k(y) = 2$ and $d(x,w_i) \leq 2^{-n+M_2}$ for $i=1,2$,
where $\{w_1,w_2\} = N_n^k(y)$ and
$M_2$ is a large constant fixed later.
Denote
$Z = Z(x) \cap Z(w_1) \cap Z(w_2)$, where
\[ Z(z) = X_{n+N_1} \cap B(x,2^{-n+M_2+3}) \backslash B(z,2^{-n-N_1}). \]
Here $N_1$ is an integer larger than $10$.
By the construction $m_{n+N_1}(v) \geq n-M_2$ for any $v \in Z$ as well as $m_n(w_1) \geq n-M_2$ for $i=1,2$.
Thus by choosing $N_0$, $K$ and $\phi_1<1$ big enough and $\varepsilon_1$ small enough depending also on $N_1$ we see as in the proof of Lemma~\ref{leordom}
that $Z \cup \{x,w_1,w_2\} \in \oom{\phi_1}$.

Assume first that there does not exist $v \in Z$ such that $xvw_1$.
In this case we content ourselves with showing that an endpoint of $\Gamma_{n+N_1}^0$ (i.e. $z\in X_{n+N_1}$ with $\#N_{n+N_1}^0(z) \leq 1$)
or a relatively long edge of $\Gamma_{n+N_1}^0$ lies close to $x$.
By Lemma~\ref{le2} we find $u \in X_{n+N_1}$ with $d(x,u) < 2^{-n+5}$.
Further we assume that there exist $u_1,u_2 \in X_{n+N_1}$ such that $\{u_1,u,u_2\}\in\om{\phi_1}$, $u_1uu_2$ and
$2  < 2^{n-M_2}d(u,u_i) \leq 4$ for $i=1,2$.
If this is not the case then by Lemma~\ref{leordom} and (H1) an endpoint of $\Gamma_{n+N_1}^0$
lies in $B(u,2^{-n+M_2+1}) \subset B(x,2^{-n+M_2+2})$ or there exists $\{y_1,y_2\} \in \E_{n+N_1}^0$ such that
$d(x,y_1) \leq 2^{-n+M_2+1} < d(y_1,y_2)$. For this we choose $M_1 \geq N_1 + M_2 + 2$.
Notice that we used (H1) for $n+N_1$ though we have not verified it yet. This will be done later.
Now $u_1,u_2 \in Z$ since
$(2^{M_2} + 1)2^{-n} \leq (2^{M_2+1} - 32)2^{-n} < d(u_i,u) - d(u,x) \leq d(u_i,x) \leq d(u_i,u) + d(u,x) < (2^{M_2+2} + 32)2^{-n} \leq 2^{-n+M_2+3}$
and
$d(u_i,w_j) \geq d(u_i,u) - d(x,w_j) - d(x,u) > (2^{M_2}-32)2^{-n} \geq 2^{-n-N_1}$
for $i,j=1,2$.
Since $16 \leq \phi_12^{M_2}$, Lemma~\ref{lemove}~(ii) gives $u_1xu_2$.
Thus $u_1w_1xu_2$ or $u_1xw_1u_2$.
Denote $Z_1 = \{\, v \in Z\, :\, vw_1x\,  \}$ and $Z_2 = \{\,  v \in Z\, :\, w_1xv\,  \}$.
Now $Z = Z_1 \cup Z_2$ and
\[ d(Z_1 \cup B(w_1,2^{-n-N_1}), Z_2 \cup B(x,2^{-n-N_1}) \cup B(w_2,2^{-n-N_1})) > d(w_1,x) - 2^{-n-N_1+1} > 2^{-n-3} \]
by \eqref{d1}.
Since $u_i \in Z_i \cap B(u,2^{-n+M_2+2}) \subset B(x,2^{-n+M_2+3})$ for $i=1,2$, we deduce by (H1) (which is not proved yet)
that there exists $\{y_1,y_2\} \in \E_{n+N_1}^0$ such that
$d(y_1,y_2) > 2^{-n-3}$
and $d(x,y_1) < 2^{-n+M_2+3}$.

We now assume that there exist
$v_1,v_2 \in Z$
such that $xv_1w_1$ and $xv_2w_2$. Recall that we still assume $\{w_1,w_2\} = N_n^k(y) \subset B(x,2^{-n+M_2})$ and $w_1\neq w_2$.
We may also suppose there exists $v_3 \in Z\backslash B(x,2^{-n+M_2})$. Namely, if such $v_3$ does not exist then an endpoint of $\Gamma_{n+N_1}^0$ or a relatively
long edge of $\Gamma_{n+N_1}^0$ lies close to $x$ as above, which is enough for us for now.

Let us choose $v_3w_1x$.
Denote $\A(z) = \A_{m_n(z)}(q_{m_n(z)}^{-1}(z))$ and
$B_i = B_{n+N_1}(v_i)$
for $z=x,w_1,w_2$ and $i=1,2,3$.
Choosing $R_1$ big enough depending on $M_2$, and $r_1$ small enough depending on $M_2$ and $N_1$, we have
\[ B_i \times B_j \subset \A(x) \cap \A(w_1) \cap \A(w_2) \]
for distinct $i,j\in\{1,2,3\}$.
Fix $\lambda > 0$ and let $\G = \G(x) \cup \G(w_1) \cup \G(w_2)$, where
\begin{align*}
  \G(z) &= \biggl\{\, (\zeta,\xi)\in\A(z)\, :\, c(\zeta,\xi,z)^2 \geq
  G\dashint_{\A(z)} c(z_1,z_2,z)^2\, d\m^2(z_1,z_2) + \lambda\, \biggr\}.
\end{align*}
Here $G$ is a large constant depending on $C_0$, $\delta$, $R_1$, $N_1$ and $N_0$.
By the Tchebychev inequality
\begin{align}
\label{tch}
\begin{split}
  \m^2(\G) &\leq \m^2(\G(x)) + \m^2(\G(w_1)) + \m^2(\G(w_2))  \\
  &\leq \frac{1}{G}\left( \m^2(\A(x)) + \m^2(\A(w_1)) + \m^2(\A(w_2)) \right)  \\
  &\leq \frac{C_0^2R_1^2\left( 2^{-2} + 2^{2M_2+1} \right)}{G4^n}.
\end{split}
\end{align}
Denote
$U_i = \{ v\in B_1\, :\, \{v\}\times B_i\subset\G\}$ for $i=2,3$.
We next show that there exists $(u_1,u_2,u_3) \in B_1\times B_2\times B_3$
such that $(u_1,u_2)\not\in\G$ and $(u_1,u_3)\not\in\G$.
Suppose this is false. Then $B_1 = U_2 \cup U_3$.
Letting
\[ p(z) = G\dashint_{\A(z)} c(z_1,z_2,z)^2\, d\m^2(z_1,z_2) + \lambda\]
we have for each $i \in \{2,3\}$ and $z \in \{x,w_1,w_2\}$
\begin{align*}
  \left\{\, v\in B_1\, :\, \{v\}\times B_i\subset\G(z)\, \right\}
  = \left\{\, v\in B_1\, :\, \text{$c(v,\xi,z)^2 \geq p(z)$ for all $\xi \in B_i$}\, \right\},
\end{align*}
which is a closed set (in $X$).
Thus $U_2$ and $U_3$ are closed and we get
\begin{align*}
  \m^2(\G) &\geq \m^2(U_2\times B_2) + \m^2(U_3\times B_3)  \\
  &= \m(U_2)\m(B_2) + \m(U_3)\m(B_3)  \\
  &\geq \left(\m(U_2)+\m(U_3)\right)\min\left\{\m(B_2),\m(B_3)\right\}  \\
  &\geq \m(B_1)\min\left\{\m(B_2),\m(B_3)\right\}
  \geq \delta^24^{-n-N_1-N_0},
\end{align*}
which contradicts \eqref{tch} provided that $G$ has been chosen big enough.

As in the proof of Lemma~\ref{leordom} we easily see that $\{x,w_1,w_2\} \cup V \in \oom{\phi_1}$
for $V = \{v_1,v_2,v_3\}$, $\{u_1,u_2,u_3\}$ if we choose $\phi_1$ big enough
depending on $M_2$, $N_0$ big enough
depending on $M_2$, $N_1$ and $\phi_1$, $K$ big enough depending on $M_2$ and $N_1$, and $\varepsilon_1$ small enough
depending on $\phi_1$, $M_2$, $N_0$, $N_1$ and $\delta$. Using the assumptions
$xv_1w_1$, $xv_2w_2$ and $v_3w_1x$ we deduce by Lemma~\ref{lemove} that $u_3w_1u_1xu_2w_2$.
Letting
\[ \varphi = -\cos\min\{\sphericalangle u_1xu_2,\sphericalangle u_3w_1u_1,\sphericalangle u_1u_2w_2,\sphericalangle u_3u_1w_2\} \]
we have
\begin{align*}
  d(w_1,w_2) &\geq d(u_3,w_2) - d(u_3,w_1)  \\
  &\geq \varphi d(u_3,u_1) + d(u_1,w_2) - d(u_3,w_1)  \\
  &\geq \varphi(\varphi d(u_3,w_1)+d(w_1,u_1)) + d(u_1,u_2) + \varphi d(u_2,w_2) - d(u_3,w_1)  \\
  &\geq \varphi(\varphi d(u_3,w_1)+d(w_1,u_1)) + d(u_1,x) + \varphi d(x,u_2) + \varphi d(u_2,w_2) - d(u_3,w_1)  \\
  &\geq \varphi(d(w_1,u_1)+d(u_1,x)+d(x,u_2)+d(u_2,w_2)) + (\varphi^2-1)d(u_3,w_1)  \\
  &\geq \varphi(d(w_1,x)+d(x,w_2)) + (\varphi^2-1)d(u_3,w_1).
\end{align*}
Denote $\lambda_1=c(x,u_1,u_2)d(u_1,u_2)$, $\lambda_2=c(w_1,u_1,u_3)d(u_1,u_3)$, $\lambda_3=c(w_2,u_1,u_2)d(u_1,w_2)$ and
$\lambda_4=c(w_2,u_1,u_3)d(u_3,w_2)$.
Since $4(1-\varphi^2) = \max\{\lambda_1,\lambda_2,\lambda_3,\lambda_4\}^2$, $(u_1,u_2)\not\in\G$, $(u_1,u_3)\not\in\G$
and $\lambda > 0$ is arbitrary, we further get by \eqref{kaava1}
\begin{equation}
\label{intest1}
\begin{split}
  l(\Gamma_n^{k+1})-l(\Gamma_n^k) &\leq d(w_1,x)+d(x,w_2)-d(w_1,w_2)  \\
  &\leq (1-\varphi)(d(w_1,x)+d(x,w_2)) + (1-\varphi^2)d(u_3,w_1)  \\
  &\leq (1-\varphi^2)(d(w_1,x)+d(x,w_2)+d(u_3,w_1))  \\
  &\leq C_2\int\limits_{B(x,R_22^{-n})}\int\limits_{A_n(z_3)}\int\limits_{A_n(z_3)\backslash B(z_2,r_12^{-n})} c(z_1,z_2,z_3)^2\, d\m z_1\, d\m z_2\, d\m z_3,
\end{split}
\end{equation}
where
$A_n(z) = B(z,R_22^{-n}) \backslash B(z,r_12^{-n})$, $R_2$ depends on $M_2$, and $C_2$ depends on $G$ and $\delta$.

Let now $k\in\{j_{n+1},\dotsc,k_{n+1}-1\}$ and assume by induction that we have
constructed a curve $\Gamma_n^k$ and the following hypothesis is satisfied:
\begin{itemize}
  \item[(H3)] If $z \in X_n^{j_{n+1}}$, $B(z,(2^{M_1}-3)2^{-n}) \cap X_n^k = \{z_1,\dotsc,z_m\}$
  and $z_1z_2\dotsc z_m$, then
  $\{ z_i,z_{i+1} \} \in \E_n^k$ for all $i \in \{1,\dotsc,m-1\}$.
\end{itemize}
Clearly (H2) implies (H3) if $k=j_{n+1}$.
We again denote $x=x_{n+1}^{k+1}$ and let $y\in X_n^k$ be such that $d(x,y) = d(x,X_n^k)$.
Denote $N'(y) = \{\, w \in N_n^k(y)\, :\, yxw\, \}$.
If $N'(y) \neq \emptyset$,
we choose $w \in N'(y)$ such that $d(y,w) = d(y,N'(y))$.
If $N'(y) = \emptyset$ and $\#N_n^k(y) = 2$, we choose $w \in N_n^k(y)$ such that $d(y,w) = \max\{\, d(y,z)\, :\, z \in N_n^k(y)\, \}$.
Then
we set
\[ \E_n^{k+1} = \left(\E_n^k \backslash \{y,w\}\right) \cup \{ \{y,x\},\{x,w\} \}. \]
If $N'(y) = \emptyset$ and $\#N_n^k(y) \leq 1$, we put
\[ \E_n^{k+1} = \E_n^k \cup \{ \{y,x\} \}. \]
Let us note that the construction does not depend on choice of $y$ by \eqref{d3}, Lemma~\ref{leordom} and (H3).

If $\#N_n^k(y) \leq 1$ and $N'(y) = \emptyset$ we will use the simple estimate
\begin{equation}
\label{est21}
  l(\Gamma_n^{k+1})-l(\Gamma_n^k) \leq d(x,y) \leq (1 + 2^{-N_0+1})2^{-n},
\end{equation}
which comes from \eqref{d3}.

We next assume that $yxw$ for some $w \in B(y,2^{-n+M_2}) \cap N_n^k(y)$.
Choosing $M_1$ large enough depending on $M_2$ and using \eqref{d3}, Lemma~\ref{leordom} and (H3) we see
that this is the case if
$\#N_n^k(y) = 2$ and $N_n^k(y) \subset B(y,2^{-n+M_2})$.
As before, 
an endpoint of $\Gamma_{n+N_1}^0$ lies in $B(x,2^{-n+M_2+2})$,
or there exists $\{y_1,y_2\} \in \E_{n+N_1}^0$ such that
$d(y_1,y_2) > 2^{-n-3}$ and $d(x,y_1) < 2^{-n+M_2+3}$,
or
\begin{equation}
\label{intest2}
  l(\Gamma_n^{k+1})-l(\Gamma_n^k)
  \leq C_2\int\limits_{B(x,R_22^{-n})}\int\limits_{A_n(z_3)}\int\limits_{A_n(z_3)\backslash B(z_2,r_12^{-n})} c(z_1,z_2,z_3)^2\, d\m z_1\, d\m z_2\, d\m z_3
\end{equation}
as in \eqref{intest1}.

Let us next show that (H3) remains valid when we replace $k$ by $k+1$.
Assume by induction that also the following condition is satisfied for any $z \in X_n^{j_{n+1}}$.
For $k=j_{n+1}$ this follows directly from (H2).
\begin{itemize}
  \item[(*)] If $w_1 \in B(z,(2^{M_1}-1)2^{-n}) \cap X_n^{j_{n+1}}$,
  $\{w_2,\dotsc,w_{p-1}\} = \{\, w \in X_n^k\, :\, w_1wz\, \}$,
  $w_p = z$ and $w_1w_2\dotsc w_p$, then $\{ w_i,w_{i+1} \} \in \E_n^k$
  for all $i \in \{1,\dotsc,p-1\}$.
\end{itemize}
We very first show that we can replace $k$ by $k+1$ in (*).
Let $z \in X_n^{j_{n+1}}$ and $\{w_1,\dotsc,w_p\}$ be as in the hypothesis of (*).
Clearly we can assume that $w_1 \neq z$.
We first notice that $w_1xz$ implies $y \in \{w_1,\dotsc,w_p\}$.
Namely, $w_1xz$ implies $d(x,z) \leq 2^{-n+M_1}$ and further
$d(x,w) > 2^{-n+M_1}$ for all $w \in X_n^k \backslash B(z,2^{-n+M_1+1})$.
Since $B(z,2^{-n+M_1+1}) \cap X_n^{k+1}$ is orderable by Lemma~\ref{leordom}, we get the conclusion.
Thus we can assume that $y \in \{w_1,\dotsc,w_p\}$. Since
\begin{equation}
\label{dxy}
  d(x,y) \leq d(x,X_n^{j_{n+1}}) \leq (1 + 2^{-N_0+1} + \vartheta)2^{-n} < 2^{-n+1}
\end{equation}
by \eqref{d3}, the set
$\{x,w_1,\dotsc,w_p\}$ is orderable by Lemma~\ref{leordom}.
If $y \in \{w_2,\dotsc,w_{p-1}\}$ then (*) is clearly
valid also for $k+1$ by the construction.
Assume now that $y=w_1$. 
By Lemma~\ref{leordom} there cannot exist $w \in X_n^k$ such that $yww_2$.
If $yxw_2$ and $\{\{y,x\},\{x,w_2\}\} \not\subset \E_n^{k+1}$, then
by the construction there is $w \in N_n^k(y) \backslash \{w_2\}$ such that $wxy$ and
$d(w,y) \leq d(y,w_2)$.
Thus the quadruple $\{w,y,x,w_2\} \subset B(y,2^{-n+M_1}) \cap X_n^{k+1}$
is not orderable contradicting Lemma~\ref{leordom}.
If $xyw_2$ and $\{y,w_2\} \not\in \E_n^{k+1}$, then
by the construction there is $w \in N_n^k(y) \backslash \{w_2\}$ such that $wyx$ and
$d(w,y) \leq d(y,w_2)$.
Thus again $\{w,x,y,w_2\}$ is not orderable.
This shows that (*) holds for $k+1$.
The case $y=z$ is treated similarly by replacing $w_2$ by $w_{p-1}$.

Let us now show that (H3) still holds if we replace $k$ by $k+1$.
If $d(y,z) \leq (2^{M_1}-3)2^{-n}$ this can be seen similarly as above.
We can clearly assume that
$B(z,(2^{M_1}-3)2^{-n}) \cap \{x,y\} \neq \emptyset$.
Let $v\in X_n^{j_{n+1}}$ be such that $d(x,v) = d(x,X_n)$.
Then $x,y,v \in B(z,(2^{M_1}-1)2^{-n})$ and $x,y \in B(v,2^{-n+2})$ by \eqref{dxy}.
So we only need consider the case $y \not\in Z$, $x \in B(z,(2^{M_1}-3)2^{-n})$
and $v \neq z \neq y$.
Here we let $Z = \{z_1,\dotsc,z_m\}$ be as in the hypothesis of (H3).
Now there do not exist $z',z'' \in Z$ such that $z'xz''$ because
$B(z,2^{-n+M_1}) \cap X_n^{k+1}$ is orderable and $y \not\in Z$.
Let us choose $yxz_1z_2\dotsc z_m$.

If $v \neq y$ then we must have $yxvz$ or $vyxz$.
If $yxvz$ then $y,z_1 \in B(v,2^{-n+2})$ and we have $\{y,z_1\} \in \E_n^k$ by (H3) for $v$
(and Lemma~\ref{leordom}).
If $vyxz$ or $v=y$ then $\{y,z_1\} \in \E_n^k$ by (*).
So in any case $\{y,z_1\} \in \E_n^k$.
If $\{x,z_1\} \not\in \E_n^{k+1}$ then
there is $w \in N_n^k(y) \backslash \{z_1\}$ such that $wxy$ and
$d(w,y) \leq d(y,z_1)$.
Since $xwz_1$ is not possible we must have $wyz_1$.
Thus the quadruple $\{w,y,x,z_1\} \subset B(y,2^{-n+M_1}) \cap X_n^{k+1}$
is contradictingly not orderable and we get (H3) for $k+1$.


Now $\Gamma_{n+1}^0$ is obtained simply by removing $D_n \backslash X_{n+1}$ from $X_n^{k_{n+1}}$
so that the order of the points in $X_{n+1}$ does not change.
Precisely, denote $D_n \backslash X_{n+1} = \{x_1,\dotsc,x_m\}$ and set inductively $X_n^{k_{n+1}+i} = X_n^{k_{n+1}+i-1} \backslash \{x_i\}$ for $i=1,\dotsc,m$.
If $\#N_n^{k_{n+1}+i-1}(x_i) = 2$ we set
\[ \E_n^{k_{n+1}+i} = \left(\E_n^{k_{n+1}+i-1} \backslash \{\, \{x_i,w_1\}\{x_i,w_2\}\, \} \right) \cup \{\, \{w_1,w_2\}\, \}, \]
where $\{w_1,w_2\} = N_n^{k_{n+1}+i-1}(x_i)$.
If $\#N_n^{k_{n+1}+i-1}(x_i) = 1$ we set
\[ \E_n^{k_{n+1}+i} = \E_n^{k_{n+1}+i-1} \backslash \{\, \{x_i,w\}\, \}, \]
where $w \in N_n^{k_{n+1}+i-1}(x_i)$.
Finally we put $\E_n^0 = \E_n^{k_{n+1}+m}$.

By induction (H3) holds for all $k \in \{j_{n+1},\dotsc,k_{n+1}\}$. 
Clearly we can also replace $X_n^{k_{n+1}}$ by $X_{n+1}$.
For $z \in X_{n+1}$ there is $y \in X_n^{j_{n+1}}$ such that
$d(z,y) < 2^{-n+1}$ by \eqref{dxy}.
Thus, since $M_1$ is chosen to be a large constant, $B(z,2^{-n-1+M_1}) \subset B(y,(2^{M_1}-3)2^{-n})$.
So we have (H1) for $n+1$, and by induction (H1) and the following simple variant of (H3) hold for each integers
$n \geq n_0$ and $j_{n+1} \leq k \leq k_{n+1}$.
\begin{itemize}
  \item[(H4)]
  If $z \in X_n^k$, $B(z,2^{-n+M_1-1}) \cap X_n^k = \{z_1,\dotsc,z_m\}$
  and $z_1z_2\dotsc z_m$, then
  $\{ z_i,z_{i+1} \} \in \E_n^k$ for all $i \in \{1,\dotsc,m-1\}$.
\end{itemize}

\begin{lemma}
\label{lelong1}
  Let $n > n_0$ and $j_n < k < m \leq k_n$.
  Assume that $d(x_n^m,x_n^k) = d(x_n^m,X_{n-1}^{m-1})$.
  Then there is (unique) $z \in B(x_n^k,2^{-n+3}) \cap N_{n-1}^{m-1}(x_n^k)$ such that $x_n^kx_n^mz$.
\end{lemma}

\begin{proof}
Let $y_i \in X_{n-1}^{k-1}$ be such that $d(x_n^i,y_i) = d(x_n^i,X_{n-1}^{k-1})$ for $i=k,m$.
Now $d(x_n^m,y_m) \leq d(x_n^k,y_k) < 2^{-n+2}$ by \eqref{xn} and \eqref{dxy}.
The set $\{x_n^k,x_n^m,y_k,y_m\}$ is orderable by Lemma~\ref{leordom}.
By the assumption either $y_mx_n^mx_n^ky_k$ or $x_n^kx_n^my_k$.
Thus the claim follows from Lemma~\ref{leordom} and (H4).
\end{proof}

\begin{lemma}
\label{lelong2}
  Let $m > n \geq n_0$ and $k \in \{0,\dotsc,k_{n+1}\}$.
  Assume that $\{x_1,x_2\} \in \E_n^k$ and $d(x_1,x_2) \geq 2^{-n+6}$.
  Then there is $\{y_1,y_2\} \in \E_m^0$ such that $d(x_i,y_i) < 2^{-n+6}$ for $i=1,2$.
\end{lemma}

\begin{proof}
For any $p \geq n_0$, $l \in \{0,\dotsc,k_{p+1}\}$ and $\{z_1,z_2\} \in \E_p^l$ with $d(z_1,z_2) \geq 2^{-p+4}$
there is $\{w_1,w_2\} \in \E_p^{k_{p+1}}$ such that
$z_i = w_i$ or $w_i \in D_{p+1}$ with
$d(z_i,w_i) < 2^{-p+1}$ for $i=1,2$.
This follows from Lemma~\ref{lelong1} and the construction.
If $d(z_1,z_2) \geq 2^{-p+6}$ then by the construction, Lemma~\ref{le2} and (H4) we find
$\{y_1,y_2\} \in \E_{p+1}^0$ such that $d(z_i,y_i) < 2^{-p+5}$ and $d(y_1,y_2) \geq d(z_1,z_2) - 2^{-p+6}$  for $i=1,2$. If $d(z_1,z_2) \leq 2^{-p+M_1-1}$ then $d(y_1,y_2) \geq d(w_1,w_2)$ by (H4).
Thus $d(y_1,y_2) \geq 2^{-p+5}$ by choosing $M_1 \geq 8$ and the claim follows by induction.
\end{proof}

Let $N_2$ be a large integer. By choosing $M_1 \geq N_2 $ we have $X = B(z,2^{-N_2-n_0+M_1})$ for $z \in X_{N_2+n_0}$,
and thus Lemma~\ref{leordom} and (H1) imply
\begin{align}
\label{lest0}
  l(\Gamma_{N_2+n_0}^0) \leq d(X)/\phi_1.
\end{align}
For $n \geq N_2+n_0$, $k \in \{1,\dotsc,k_{n+1}\}$ we let $p_n^k \in X_{n-1}^{k-1}$ be such that
$d(x_n^k,p_n^k) = d(x_n^k,X_{n-1}^{k-1})$.
Let now $m > N_2+n_0$.

Denote
\begin{align*}
\Lambda_n^1 &= \{\, k \in \{1,\dotsc,j_n\}\, :\, \text{$\#N_{n-1}^{k-1}(p_n^k) = 2$ and $N_{n-1}^{k-1}(p_n^k) \subset B(x_n^k,2^{-n+M_2+1})$}\, \}, \\
\Lambda_n^2 &= \{\, k \in \{j_n+1,\dotsc,k_n\}\, :\, \text{$p_n^kx_n^kw$ for $w \in B(p_n^k,2^{-n+M_2+1}) \cap N_{n-1}^{k-1}(p_n^k)$}\, \}.
\end{align*}
By (H4) (and \eqref{d3})
\begin{align*}
  \sum_{k \in \Lambda_{n+1}(z,M)} \left(l(\Gamma_n^k) - l(\Gamma_n^{k-1}) \right)
  \leq (\phi_1^{-1}-1)(M+2^{M_2+1})2^{-n}
\end{align*}
for any $n < m$, $z \in X_m$ and $M \leq 2^{M_1 - 2}$, where
\[ \Lambda_n(z,M) = \{\, k \in \Lambda_n^1 \cup \Lambda_n^2\, :\, d(x_n^k,z) \leq M2^{-n}\, \}. \]
Here we use $M_1 \geq M_2 + 6$.
Using this, Lemma~\ref{lelong2}, \eqref{d1}, \eqref{intest1} and \eqref{intest2} we get
\begin{align}
\label{lest12}
\begin{split}
  &\sum_{n=N_2+n_0}^{m-1} \sum_{k\in\Lambda_{n+1}^1 \cup \Lambda_{n+1}^2} \left(l(\Gamma_n^k)-l(\Gamma_n^{k-1})\right) \\
  &\leq \sum_{n=N_2+n_0}^{m-1} (\phi_1^{-1}-1)(2^{M_2+2}+2^{M_2+1})2^{-n+1} + (\lambda_1 + \lambda_2)l(\Gamma_m^0) \\
  &+ C_2\sum_{n=N_2+n_0}^{m-1} \sum_{x \in D_{n+1}} \int\limits_{B(x,R_22^{-n})}\int\limits_{A_n(z_3)}\int\limits_{A_n(z_3)\backslash B(z_2,r_12^{-n})} c(z_1,z_2,z_3)^2\, d\m z_1\, d\m z_2\, d\m z_3,
\end{split}
\end{align}
where
\begin{align*}
  \lambda_1 &= \frac{4(\phi_1^{-1}-1)(2^{M_2+3}+2^{-N_1+6}+2^{M_2+1})}{2^{-3} - 2^{-N_1+7}}, \\
  \lambda_2 &= \sum_{n=m-N_1+1}^{m-1} (\phi_1^{-1}-1)(32+2^{M_2+1})2^{-n+m+2} < (\phi_1^{-1}-1)(32+2^{M_2+1})2^{N_1+2}.
\end{align*}
By Lemma~\ref{leordom} and $\eqref{d1}$ we have
\begin{equation}
\label{lest13}
  \#(B(z,M2^{-n}) \cap D_{n+1}) \leq 8M\phi_1^{-1} + 1
\end{equation}
for any $n \geq n_0$, $z \in X$ and $M \leq 2^{M_1+1}$.
Furthermore, if $i < j$ and $A_i(z) \cap A_j(z) \neq \emptyset$, then $r_12^{-i} < R_22^{-j}$ which gives
$j - i < (\log R_2 - \log r_1)/\log 2$. Thus by choosing $2^{M_1+1} \geq R_2$ and $K \geq R_2/R_1 + 1$
\begin{align*}
  &\sum_{n=N_2+n_0}^{m-1} \sum_{x \in D_{n+1}} \int_{B(x,R_22^{-n})}\int_{A_n(z_3)}\int_{A_n(z_3)\backslash B(z_2,r_12^{-n})} c(z_1,z_2,z_3)^2\, d\m z_1\, d\m z_2\, d\m z_3 \\
  &\leq C_3 \int_X \sum_{n=N_2+n_0}^{m-1}
  \int_{A_n(z_3)}\int_{\mathcal{T}_K(X)_{(z_2,z_3)}} c(z_1,z_2,z_3)^2\, d\m z_1\, d\m z_2\, d\m z_3  \\
  &\leq C_3C_4 \int_X\int_X\int_{\mathcal{T}_K(X)_{(z_2,z_3)}} c(z_1,z_2,z_3)^2\, d\m z_1\, d\m z_2\, d\m z_3  \\
  &= C_3C_4 c^2_K(X,\m),
\end{align*}
where $C_3 = 8R_2\phi_1^{-1} + 1$ and $C_4 = (\log R_2 - \log r_1)/\log 2$.
Thus by \eqref{lest12} and (iii)
\begin{align}
\label{lest1}
\begin{split}
  &\sum_{n=N_2+n_0}^{m-1} \sum_{k\in\Lambda_{n+1}^1 \cup \Lambda_{n+1}^2} \left(l(\Gamma_n^k)-l(\Gamma_n^{k-1})\right) \\
  &< (\phi_1^{-1}-1)2^{M_2+5-N_2-n_0} + (\lambda_1 + \lambda_2)l(\Gamma_m^0) + C_2C_3C_4\varepsilon_0d(X).
\end{split}
\end{align}

By \eqref{est11}, \eqref{est21} and Lemma~\ref{lelong1}
\begin{align}
\label{lest3}
  \sum_{n=N_2+n_0}^{m-1} \sum_{k \in \Lambda_{n+1}^3} \left(l(\Gamma_n^k) - l(\Gamma_n^{k-1}) \right)
  \leq \sum_{n=N_2+n_0}^\infty  (\vartheta + 1 + 2^{-N_0+1})2^{-n+1} < 2^{-N_2-n_0+3},
\end{align}
where
$\Lambda_n^3 = \{\, k \in \{1,\dotsc,k_n\} \backslash \Lambda_n^2\, :\, \#N_{n-1}^{k-1}(p_n^k) \leq 1\, \}$.
Further by \eqref{d3}, \eqref{lest13} and Lemma~\ref{lelong2}
\begin{align}
\label{lest4}
  \sum_{n=N_2+n_0}^{m-1} \sum_{k \in \Lambda_{n+1}^4} \left(l(\Gamma_n^k) - l(\Gamma_n^{k-1}) \right)
  \leq \frac{16(528\phi_1^{-1} + 1)l(\Gamma_m^0)}{2^{M_2} - 128},
\end{align}
where
$\Lambda_n^4 = \{\, k \in \{1,\dotsc,k_n\} \backslash \Lambda_n^2\, :\, N_{n-1}^{k-1}(p_n^k) \not\subset B(p_n^k,2^{-n+M_2+1})\, \}$

Choosing $M_2$ large enough depending on $\tau_0$, $\phi_1 < 1$ large enough depending on $M_2$ and $\tau_0$,
$N_2$ large enough depending on $M_2$ and $\tau_0$, and $\varepsilon_0$ small enoug depending on
$C_2C_3C_4$ and $\tau_0$,
\begin{align}
\label{lGammam}
  l(\Gamma_m^0) \leq (1 + \tau_0)d(X)
\end{align}
by \eqref{lest0}, \eqref{lest1}, \eqref{lest3} and \eqref{lest4}.

For any $n >  n_0$ let $f_n : [0,1] \to \NS$ be $(1 + \tau_0)d(X)$-Lipschitz function with $f_n([0,1]) = \Gamma_n^0$.
Since each $X_n$ is finite, we easily see by \eqref{d3} that the closure of $\bigcup_{n=0}^\infty X_n \subset \NS$, denoted by $Y$, is compact.
Thus also
$S(Y) := \{\, tz_1 + (1-t)z_2 :\, \text{$z_1,z_2 \in Y$, $0 \leq t \leq 1$}\, \}$
is compact (see \cite[Lemma 5.1]{MR2337487}).
Since now $\Gamma_n^0 \subset S(Y)$ for each $n$,
we find by the Ascoli-Arzela theorem a $(1 + \tau_0)d(X)$-Lipschitz function $f:[0,1] \to \NS$, which is the uniform limit of some subsequence of $(f_n)$.
We denote $\Gamma = f([0,1])$.

\section{Size of $X \backslash \Gamma$}

Let us assume for simplicity that $f_n$ converges to $f$ uniformly.
Denote $V= X\backslash\Gamma$. In this section our goal is to show that $\m(V)  \leq \tau_0d(X)$.
We will denote by $U(x,r)$ the closed ball in $\NS$ with center $x\in \NS$ and radius $r>0$.
In this section $\Hm$ denotes the 1-dimensional Hausdorff measure on $\NS$.
Recall that
\begin{align}
\label{kaava2}
  \Hm(\Gamma) \leq (1+\tau_0)d(X).
\end{align}
Notice that we can clearly assume that $\tau_0$ is small. Set
\[ V_1 = \left\{\, z \in V\, :\, \text{$\m(B(z,r)) \leq \tau_0r/20$ for some $r \in [ 6d(z,\Gamma),d(X) ]$}\, \right\}. \]
For each $z \in V_1$ choose $r(z) \in [ 6d(z,\Gamma),d(X) ]$ such that $\m(B(z,r(z))) \leq \tau_0r(z)/20$
and let $w(z) \in \Gamma$ be so that $d(z,w(z) )= d(z,\Gamma)$.
Denote $U_z = U(w(z),r(z)/6)$ and $5B_z =  B(w(z),5r(z)/6)$. Now $z \in U_z$ and $5B_z \subset B(z,r(z))$.
By the $5r$-covering lemma we find a countable set $V_1' \subset V_1$
such that $V_1 \subset \bigcup_{z\in V_1'} 5B_z$ and the family $\{\, U_z\, :\, z \in V_1'\, \}$ is disjoint, and we get
\begin{equation}
\label{V11}
  \m(V_1) \leq \sum_{z\in V_1'} \m(5B_z) \leq \sum_{z\in V_1'} \m(B(z,r(z))) \leq \frac{\tau_0}{20} \sum_{z\in V_1'} r(z).
\end{equation}
Assume first that $\Gamma$ leaves each $U_z$, $z \in V_1'$. Then $\Hm(\Gamma \cap U_z) \geq r(z)/6$ for each $z \in V_1'$.
Thus \eqref{V11}, the disjointness of the balls $U_z$ and \eqref{kaava2}
yield
\begin{equation}
\label{V1}
  \m(V_1) \leq \frac{3\tau_0}{10} \sum_{z\in V_1'} \Hm(\Gamma \cup U_z) \leq \frac{3\tau_0}{10} \Hm(\Gamma)
  \leq \frac{\tau_0d(X)}{3}.
\end{equation}
If $\Gamma \subset U_{z_0}$ for some $z_0 \in V_1'$ then $V_1' = \{z_0\}$ by the disjointness of the balls $U_z$, $z \in V_1'$,
and \eqref{V11} gives $\m(V_1) \leq \tau_0r(z_0)/20 \leq \tau_0d(X)/20$.

Denote $H = \bigcup_{n=n_0}^\infty H_n$, $2B(y) = B(q_m^{-1}(y), 2^{-m(y)+4})$ for $y \in H$.
We next estimate the measure of the set 
\[ V_2 = \bigcup_{y \in H} 2B(y). \]
By the construction the balls $U_y := U(y, 2^{-m(y)-2})$, $y \in H$, are disjoint.
Assuming that $\Gamma$ leaves each $U_y$, $y \in H$, we thus have (by \eqref{kaava2})
\begin{align*}
  \sum_{y\in H} 2^{-m(y)} \leq 4\sum_{y\in H} \Hm(\Gamma \cap U_y) \leq 4\Hm(\Gamma) \leq 5d(X).
\end{align*}
If $\Gamma \subset U_{y_0}$ for some $y_0 \in H$ then $H = \{y_0\}$ by the disjointness of the balls $U_y$, $y \in H$,
and we have
\begin{align*}
  \sum_{y\in H} 2^{-m(y)} = 2^{-m(y_0)} \leq 2^{-n_0} < 2d(X).
\end{align*}
Using these estimates and \eqref{Hdef} we get
\begin{equation}
\label{V2}
  \m(V_2) \leq \sum_{y \in H} \m\bigl(2B(y) \backslash \bigcup_{z \in H_{m(y)}} B(z)\bigr) \leq \sum_{y \in H} C_1\delta 2^{-m(y)} \leq 5C_1\delta d(X).
\end{equation}

Now let $z \in V \backslash (V_1 \cup V_2)$. Let $n(z)$ be the integer such that
\begin{equation}
\label{nz}
  2^{-n(z)+M_1} \leq d(z,\Gamma) < 2^{-n(z)+M_1+1}.
\end{equation}
Set $D(z) = B(z,6d(z,\Gamma))$.
If $6d(z,\Gamma) > d(X)$ then
by choosing $\varepsilon_1$ and $\delta$ small enough depending on $N_0$, $M_1$ and $\mu_0$ (and using Lemma~\ref{led}) we find
that $D(z) \cap \DP_{n(z)} \neq \emptyset$.
Else, since $z \not\in V_1$, by choosing $\varepsilon_1$ small enough depending on $N_0$, $M_1$ and $\tau_0$ and then using Lemma~\ref{led}
we find $w \in D(z)$ such that
\[  \m(B(w,2^{-n(z)-N_0})) > \eta\tau_02^{-n(z)+M_1}, \]
where $\eta > 0$ depends on $N_0$ and $M_1$.
Since $z \not\in V_1$ we have
\[ \m(B(w,2^{-m})) \geq \m(B(z,2^{-m-1})) > \tau_02^{-m-1}/20 \]
for all $n_0 \leq m \leq n(z) - M_1 - 5$. So by choosing
$\delta$ small enough depending on $N_0$, $M_1$ and $\tau_0$ we get that $w \in \DP_{n(z)}$.
Now $d(w,D_{n(z)}) \leq 2^{-n(z)+1}$ or $w \in B(y)$ for some $y \in H_{n(z)}$. In the latter case
$d(z,w) > 2^{-m(y) + 3}$, because $z \not\in V_2$.
Thus in both cases $d(w,X_{n(z)}) < 7d(z,\Gamma)$ and further $d(z,X_{n(z)}) < 13d(z,\Gamma)$.
Let $y(z) \in X_{n(z)}$ be such that $d(z,y(z)) = d(z,X_{n(z)})$. By Lemma~\ref{le2} and \eqref{nz} we have 
\begin{equation}
\label{dzy}
  2^{-n(z)+M_1-1} < d(z,y(z)) < 2^{-n(z)+M_1+5}.
\end{equation}

For $x \in X$ and $n \geq n_0$ we denote
\[ W_n(x) = B(x,2^{-n+M_1+6}) \cap \{\, z \in V \backslash (V_1 \cap V_2)\, :\, n(z) = n\, \}. \]

\begin{lemma}
\label{lemW}
  It holds that $\mu(W_n(x)) \leq \tau_02^{-n}/20$ for all $n \geq n_0$ and $x \in X$.
\end{lemma}

\begin{proof}
Let $n \geq n_0$ and $x \in X$. Suppose to the contrary that $\m(W_n(x)) > \tau_02^{-n}/20$.
By choosing $\varepsilon_1$ small enough depending on $N_0$, $M_1$ and $\tau_0$ and then using Lemma~\ref{led}
we find $z \in W_n(x)$ such that
$\m(B(z,2^{-n-N_0})) > \eta\tau_02^{-n}$,
where $\eta > 0$ depends on $N_0$ and $M_1$.
As before, since $z \not\in V_1$ we have
$ \m(B(z,2^{-m})) > \tau_02^{-m}/20$
for all $n_0 \leq m \leq n - M_1 - 4$.
So by choosing
$\delta$ small enough depending on $N_0$, $M_1$ and $\tau_0$ we get that $z \in \DP_n$.
Since $z \not\in V_2$, we have $d(z,D_n) \leq 2^{-n+1}$ which contradicts \eqref{dzy}.
\end{proof}

Denote $V_3 = \{\, z \in V\backslash (V_1 \cup V_2)\, :\, N_{n(z)}^0(y(z)) \leq 1\, \}$.
By \eqref{dzy} and Lemma~\ref{lemW}
\begin{equation}
\label{V3}
  \m(V_3) < \sum_{n=n_0}^\infty \tau_02^{-n+1}/20 = \tau_02^{-n_0+2}/20 < \tau_0d(X)/2.
\end{equation}

Set
\[ V_4 = \left\{\, z \in V\backslash (V_1 \cup V_2)\, :\, \text{$\{z,v,w\} \not\in \om{\phi_1}$ for some $v,w \in Z(z)$}\, \right\}, \]
where $Z(z) = B(y(z),2^{-n(z)+M_1}) \cap X_{n(z)}$.
If $\#Z(z) \geq 2$ then $m_{n(z)}(v) \geq n(z) - M_1$ for each $v \in Z(z)$.
Let $z \in V_4$ and choose $v_1,v_2 \in Z(z)$ with $\{z,v_1,v_2\} \not\in \om{\phi_1}$.
As before, by choosing $N_0$ large enough depending on $\phi_1$ and $M_1$ we deduce that
$c(z,w_1,w_2) \geq c2^{-n(z)}$ for all $w_i \in B_{n(z)}(v_i)$, $i=1,2$, where $c$ is a positive constant depending on
$N_0$, $\phi_1$ and $M_1$. Thus
\[ \int_{B_{n(z)}(v_1) \times B_{n(z)}(v_2)} c(z,w_1,w_2)^2\, d\m^2(w_1,w_2) \geq c^2 \delta^2 2^{-2N_0}. \]
Choosing the constant $K$ large enough depending on $M_1$ we have that
$B_{n(z)}(v_1) \times B_{n(z)}(v_2) \subset \TK(X)_{z}$ and further
\begin{equation}
\label{V4}
  \m(V_4) \leq c^{-2}\delta^{-2} 4^{N_0} c^2_K(X,\m) \leq c^{-2}\delta^{-2} 4^{N_0}\varepsilon_0d(X).
\end{equation}

Let now $z \in V_5$, where $V_5 = V\backslash (V_1 \cup V_2 \cup V_3 \cup V_4)$.
We first show that $d(y(z),w) > 2^{-n(z)+M_1-1}$ for some $w \in N_{n(z)}^0(y(z))$.
Denote $N_{n(z)}^0(y(z)) = \{u,v\}$ and assume that $\{u,v\} \subset B(y(z),2^{-n(z)+M_1})$.
Recall that $\#N_{n(z)}^0(y(z)) = 2$ since $z \not\in V_3$.
Now $\{u,v,y(z)\} \in \om{\phi_1}$ and $uy(z)v$
by Lemma~\ref{leordom} and (H1).
Since $z \not\in V_4$ we further have $\{z,u,v,y(z)\} \in \om{\phi_1}$.
Choosing $\phi_1$ big enough depending on $M_1$, assuming $d(z,u) \leq d(z,v)$ and using \eqref{d1}, \eqref{dzy} and Lemma~\ref{lef}
we conclude $uzy(z)v$
and $d(u,y(z)) > d(z,y(z)) > 2^{-n(z)+M_1-1}$. So in each case we may choose $u(z) \in N_{n(z)}^0(y(z))$ such that
$d(u(z),y(z)) > 2^{-n(z)+M_1-1}$.
By Lemma~\ref{lelong2} we find Cauchy sequences $(u_n(z))_{n}$ and $(y_n(z))_{n}$ so that
$\{u_n(z),y_n(z)\} \in \E_n^0$ and $d(u(z),u_n(z)), d(y(z),y_n(z)) < 2^{-n(z)+6}$ for all $n \geq n(z)$.
By taking $M_1 \geq 9$
\begin{equation}
\label{dI}
  d(u_n(z),y_n(z)) > 2^{-n(z)+M_1-2}
\end{equation}
for all $n \geq n(z)$.

For $n \geq n_0$ and $e \in \E_n^0$ we denote
$V_5^n = \{\, z \in V_5\, :\, n(z) \leq n\, \}$ and
\[V_5^n(e) = \{\, z \in V_5^n\, :\, \{u_n(z),y_n(z)\} = e\, \}. \]

\begin{lemma}
\label{lemVne}
  It holds that $\m(V_5^n(e)) \leq \tau_0d(a,b)/20$ for each $n \geq n_0$ and $e = \{a,b\} \in \E_n^0$.
\end{lemma}

\begin{proof}
Let $n \geq n_0$ and $e = \{a,b\} \in \E_n^0$.
By \eqref{dzy} and \eqref{dI} for any $z \in V_5^n(e)$
\begin{gather*}
  d(z,\{a,b\}) < d(z,y(z)) + 2^{-n(z)+6} < 2^{-n(z)+M_1+6}, \\
  2^{-n(z)} < 2^{2-M_1}d(a,b).
\end{gather*}
Thus Lemma~\ref{lemW} gives
$\m(V_5^n(e)) \leq \tau_02^{4-M_1}d(a,b)/20$.
\end{proof}

Lemma~\ref{lemVne} and \eqref{lGammam} now imply that
$\mu(V_5^n) \leq \tau_0l(\Gamma_n^0)/20 \leq \tau_0d(X)/10$ for all integers $n \geq n_0$.
Hence
\begin{equation}
\label{V5}
  \m(V_5) \leq \tau_0d(X)/10.
\end{equation}

Combining \eqref{V1}, \eqref{V2}, \eqref{V3}, \eqref{V4} and \eqref{V5}, choosing $\delta$ small enough depending on
$C_1$ and $\tau_0$, and $\varepsilon_0$ small enough depending on $4^{N_0}/(c\delta)^2$ and $\tau_0$, we obtain
$\m(V) \leq \tau_0d(X)$.

\section*{Acknowledgements}

The author would like to thank Guy David for the possibility to use his unpublished notes as the basis for this work.
Thanks are also due to Raanan Schul and Pertti Mattila for helpful conversations.

\bibliographystyle{plain}
\bibliography{viitteet}

\end{document}